\documentclass[12pt]{article}

\usepackage{amsmath, amssymb,epsfig,bm,amsthm}
\usepackage{latexsym}
\usepackage{algorithmicx}
\usepackage[ruled,vlined,linesnumbered]{algorithm2e}

\usepackage{graphicx}
\usepackage{caption,subcaption}

\usepackage{color}
\usepackage{appendix}

\numberwithin{equation}{section}
\usepackage[margin=1.0in]{geometry}    
\usepackage{mathtools}
\mathtoolsset{showonlyrefs}

\def\m{\mbox}   
\newcommand{\q}{\quad}    \def\R{{\mathbb R}}

\newtheorem{lemma}{Lemma}
\newtheorem{prop}[lemma]{Proposition}

\newtheorem{theorem}[lemma]{Theorem}

\numberwithin{lemma}{section}

\begin{document}
\title{Stochastic convergence of regularized solutions for backward heat conduction problems}

\author{Zhongjian Wang\thanks{Division of Mathematical Sciences, School of Physical and Mathematical Sciences, Nanyang Technological University, 21 Nanyang Link, Singapore 637371. (zhongjian.wang@ntu.edu.sg).
}
\and
Wenlong Zhang\thanks{Corresponding author. Department of Mathematics $\&$  National Center for Applied Mathematics
Shenzhen, Southern University of Science and Technology (SUSTech), 1088 Xueyuan Boulevard, University Town of Shenzhen, Xili, Nanshan, Shenzhen, Guangdong Province, P.R.China. (zhangwl@sustech.edu.cn).
}
\and
Zhiwen Zhang\thanks{Corresponding author. Department of Mathematics, The University of Hong Kong, Pokfulam Road, Hong Kong SAR, P.R.China. (zhangzw@hku.hk).
}}

\date{}
\maketitle

\begin{abstract}
In this paper, we study the stochastic convergence of regularized solutions for backward heat conduction problems. These problems are recognized as ill-posed due to the exponential decay of eigenvalues associated with the forward problems. We derive an error estimate for the least-squares regularized minimization problem within the framework of stochastic convergence. Our analysis reveals that the optimal error of the Tikhonov-type least-squares optimization problem depends on the noise level, the number of sensors, and the underlying ground truth. Moreover, we propose a self-adaptive algorithm to identify the optimal regularization parameter for the optimization problem without requiring knowledge of the noise level or any other prior information, which will be very practical in applications. We present numerical examples to demonstrate the accuracy and efficiency of our proposed method. These numerical results show that our method is efficient in solving backward heat conduction problems.

\end{abstract}



\section{Introduction}

Inverse problems associated with parabolic equations have attracted considerable attention in both mathematics and engineering research fields \cite{Kaipio:2005,hasanouglu2021introduction}. In general, inverse problems can be categorized into three types: identifying physical parameters or source terms in the PDEs; determining the system's initial state; and determining the boundary conditions. In this paper, we focus on the second type of inverse problem within the context of heat transfer. Specifically, we aim to determine the initial condition from transient temperature measurement at the final time $T$. This problem is commonly referred to as the backward heat conduction problem (BHCP).

The main difficulty in solving the BHCP arises from the exponential decay of forward solutions of the heat equations with respect to the initial data. Specifically, this type of problem is considered ill-posed in the sense of Hadamard. We refer the interested reader to the papers \cite{new6,lavrent_ev1986ill,Kabanikhin2008survey} for a comprehensive review of the definitions and examples of inverse and ill-posed problems.  

The BHCP exemplifies an severely ill-posed problem that is not solvable using traditional numerical methods and requires special techniques to be employed, which have been long-standing computational challenges. To point out, we are not trying to fill this gap to solve this severely ill-posed inverse problem ultimately in this paper. We will treat this problem in a different way and try to give the theoretical and numerical analysis from the stochastic point of view.  Conditions that render the BHCP well-posed have been investigated in \cite{clark1994quasireversibility,denche2005modified}. These studies introduced supplementary hypotheses, constraining the class of functions to which the solution must belong. However, in practical settings, verifying the hypotheses for initial conditions is very difficult. Consequently, numerical methods with less stringent assumptions regarding initial conditions for solving BHCPs prove to be more beneficial.

In response to these challenges, many regularization techniques have been developed for solving BHCPs. For example, Sobolev error estimates and a prior parameter selection for semi-discrete Tikhonov regularization were obtained in \cite{Krebs2009backward}. A backward problem for the one-dimensional heat conduction equation, with the measurements on a discrete set, was considered in \cite{cheng2020backward}, and the uniqueness of recovering the initial value was proved using the analytic continuation method. To address the ill-posedness, discrete Tikhonov regularization with the generalized cross-validation rule was employed to obtain a stable numerical approximation of the initial value. It is worth noting that in \cite{Muniz1999backward}, a comparison of various inverse methods for estimating the initial condition of the heat equation was studied, demonstrating that explicit approaches to BHCP yield disastrous results unless some form of regularization is utilized.

Various other approaches have been proposed for solving BHCPs. Perturbation-based methods were introduced in \cite{Qian2007backward,Trong2007backward}, where the operator is replaced with a perturbed higher-order one that exhibits improved invertibility properties. A homotopy-based iterative regularizing scheme was proposed for solving BHCP in \cite{liu2018solving}. This approach provided error analysis for the regularizing solution with noisy measurement data and demonstrated that the algorithm is easily implementable with low computational costs. In \cite{chiwiacowsky2003different}, different approaches, such as the conjugate gradient method with the adjoint equation, regularized solution using a quasi-Newton method, and regularized solution via the genetic algorithm, were compared for solving BHCPs.

In this paper, we study a practical scenario in which observational measurements are collected point-wise from a set of distributed sensors located at $\{x_i\}^n_{i=1}$ across the physical domain \cite{AB05, L08,nelson20, NNR98}. Each sensor is subject to independent additional noise or random error due to natural noise, measurement errors, and other uncertainties in the model. We aim to adopt a realistic approach for solving such inverse problems by optimizing the mean-square error using appropriate Tikhonov-type regularizations \cite{Chen-Zhang2022, GER83, WY10}. To accomplish this,  we utilize classical methods such as regression methods, linear and nonlinear programming methods \cite{GER83}, linear and nonlinear conjugate gradient methods \cite{AB05, WY10}, and Newton-type methods during the optimization process.

Then, we investigate the stochastic convergence of the proposed method for two general types of random variables. In \cite{Chen-Zhang,Chen-Zhang2022, WANG2023112156}, the authors develop the tools to study the stochastic convergence of the numerical method for an inverse source problem. Here, we derive the convergence analysis of the BHCP with Tikhonov regularization. Compared to the inverse source problem of the parabolic equation, the backward problem is severely ill-posed due to the exponential decay of the eigenvalues \cite{new6}. Our analysis in Theorem \ref{main-thm} demonstrates that the optimal error of the Tikhonov type least-squares optimization problem depends on the noise level, the number of sensors, and the ground truth. Furthermore, we describe the asymptotic behavior of the regularization parameter with respect to these model configurations. Unlike the traditional estimate of the regularization parameter, we show that this parameter should also vary with different final times $T$, and we point out this dependence.

Leveraging these asymptotic estimates, we design a self-consistent algorithm to determine the optimal regularization parameter for solving BHCPs without knowing any prior information about the inverse problem. Finally, we carry out numerical experiments to demonstrate the accuracy and efficiency of the proposed method. Several kinds of source functions, including both smooth and discontinuous ones, are considered. From the numerical results, we observe that the proposed method successfully recovers the initial conditions for different types of source functions. The algorithm demonstrates robust performance, even in the presence of measurement noise and discontinuities. Moreover, the self-consistent approach for selecting the optimal regularization parameter proves to be effective in balancing the trade-off between fitting the data and minimizing the impact of the ill-posedness. In summary, these results validate the theoretical convergence analysis and demonstrate the practical applicability of our method in solving BHCPs with noisy measurements from distributed sensors.

The rest of the paper is organized as follows. In Section 2, we introduce the setting of the backward problem of parabolic equations. Section 3 presents the stochastic convergence analysis for the backward problem of parabolic equations. More details about the implementation of the proposed method will be discussed in Section 4.  Section 5 presents numerical results to demonstrate the accuracy of our method. Finally, in Section 6, we provide concluding remarks.

\section{Backward heat conduction problems}
To start with, we consider an initial value heat equation as follows: 
\begin{equation}\label{zz0}
	\left\{
	\begin{aligned}
		u_t +\mathcal{L}u &= 0 &\mbox{in } \Omega\times (0, T), \\
		u(x, t)&= 0  &\mbox{on } \partial \Omega\times (0, T),\\
		u(x, 0)&= f(x) &\mbox{in } \Omega,
	\end{aligned} 
	\right.
\end{equation}
where $\Omega\subset \mathbb R^d$ $(d=1,2,3)$ is a bounded domain with $C^2$ boundary or a convex domain satisfying the uniform cone condition, $\mathcal{L}$ is a second-order elliptic operator given by 
\begin{align}
    \mathcal{L}u=-\nabla\cdot (a(x)\nabla u) +c(x)u,
\end{align} and
$f(x)$ is the initial condition. We assume the  elliptic operator $\mathcal{L}$ is uniform elliptic, i.e. there exist $a_{\min}, a_{\max}>0$, such that $a_{\min}<a(x)<a_{\max}$ for all $x \in \Omega$. Moreover, we assume $a(x)\in C^{1}(\bar{\Omega})$, $c(x)\in C(\bar\Omega)$ and $c(x)\geq 0$. 

Let $u$ be the solution of the heat equation \eqref{zz0}. We define the forward operator $\mathcal{F}_{t}:$ $L^2(\Omega)\rightarrow H^2(\Omega)$ by $\mathcal{F}_tf=u(\cdot,t)$. Then the forward problem is to compute the solution $u(\cdot,t)$ for $t>0$ with known initial condition $f(x)$. 

In the backward heat conduction problem, we aim to reconstruct an unknown initial condition $f(x)\in L^2(\Omega)$ based on the final time measurement $u(\cdot,T)$ for a given final time $T>0$. Specifically, we focus on a very practically physical scenario, where we assume that observational measurements are collected point-wisely over a set of distributed sensors located at $\{x_i\}^n_{i=1}$ over the physical domain $\Omega$ (see e.g. \cite{new2, new6, L08, NNR98}). Additionally, taking into account of uncertainty in data, we also assume that the measurement data is always blurred with noise and is represented as $m_i=\mathcal{F}_Tf^*(x_i)+e_i$, $i=1, \cdots, n$, where $f^*\in L^2(\Omega)$ is the true initial condition and $\{e_i\}^n_{i=1}$ are independent random variables on a suitable probability space with zero means, i.e. $\mathbb{E}[e_i]=0$. Both of the assumptions are very practical in applications, and the analysis in this paper is rare in the mathematical field.

First, we should verify that $\mathcal{F}_Tf^*(x)$ is well-defined point-wisely to ensure the validity of the aforementioned considerations. Let $\mathcal{L}_0$ denote the operator of $\mathcal{L}$ with zero Dirichlet boundary condition. According to the analytic semigroup theory \cite{renardy2006introduction,Pazy:83}, $\mathcal{F}_t=e^{-\mathcal{L}_0t}$ is an analytic semigroup, and the second-order elliptic operator $-\mathcal{L}_0$ is the infinitesimal generator of $e^{-\mathcal{L}_0t}$. Moreover, we have the following estimate: there exists a constant $\omega$ such that,
\begin{equation}
\|\mathcal{L}_0e^{-\mathcal{L}_0t}\|\leq Ce^{\omega t}\frac{1}{t}, ~ t>0.
\end{equation}
As a consequence, $\|u(\cdot,T)\|_{H^2(\Omega)}\leq\|\mathcal{L}_0u(\cdot,T)\|_{L^2(\Omega)}\leq  C_T\|f^*\|_{L^2(\Omega)}$, i.e. $\|\mathcal{F}_Tf^*\|_{H^2(\Omega)}\leq C_T\|f^*\|_{L^2(\Omega)}$, where the constant $C_T$ only depends on $T$. According to the embedding theorem of Sobolev spaces, we know that $H^2(\Omega)$ is continuously embedded into $C(\bar\Omega)$, so that $\mathcal{F}_Tf^*(x)$ is well-defined point-wisely for all $x\in \bar\Omega$. 

Without loss of generality, we assume that the scattered locations $\{x_i\}_{i=1}^n$ are uniformly distributed
in $\Omega$, i.e., there exists a constant $B>0$ such that ${d_{\max}}/{d_{\min}} \leq B$, where 
${d_{\max}}$ and ${d_{\min}}$ are defined by 
\begin{align}\label{aa}
d_{\max}=\mathop {\rm sup}\limits_{x\in \Omega} \mathop {\rm inf}\limits_{1 \leq i \leq n} |x-x_i|  
~~~\mbox{and} ~~ ~
d_{\min}=\mathop {\rm inf}\limits_{1 \leq i \neq j \leq n} |x_i-x_j|.
\end{align}
We denote the inner product between the measurement data and any function $v\in C(\bar\Omega)$ by $(m,v)_n=\sum^n_{i=1}w_{i}^{2}m_iv(x_i)$, with weight $ w_{i}^{2}=\mathcal{O}\left(\frac{1}{n}\right)$. Moreover, we represent the inner product between two functions as $(u,v)_n=\sum^n_{i=1}w_{i}^{2}u(x_i)v(x_i)$ for any $u,v\in C(\bar\Omega)$, and the empirical norm $\|u\|_n=\big(\sum_{i=1}^{n} w_{i}^{2}u^2(x_i)\big)^{1/2}$ for any $u\in C(\bar\Omega)$.

With these definitions in place, we formulate the backward heat conduction problem as recovering the unknown initial condition $f^*$ from the noisy final time measurement data $m_i=\mathcal{F}_Tf^*(x_i)+e_i$, $i=1,...,n$. To solve this backward problem, we employ a regularization method. Specifically, we seek an approximate solution of the true initial condition $f^*$ by solving the following least-squares regularized minimization problem: 
\begin{align}\label{para-ori}
\mathop {\rm min}\limits_{f\in X}         \|\mathcal{F}_Tf-m\|^2_n+\lambda \|f\|_{L^2(\Omega)}^2\,,
\end{align}
where $\lambda$ is the regularization parameter. Two critical issues arise when implementing this method: selecting the optimal regularization parameter and the converges in probability space. In the next section, we will present theoretical results addressing stochastic convergence.

\section{Stochastic convergence of the inverse problem}\label{sec:analysis}
To show the stochastic convergence of the regularization method \eqref{para-ori}, we first revisit an important property of the eigenvalue distribution for the elliptic operator $\mathcal{L}$, as presented in \cite{agmon,Fleckinger}. 


\begin{prop}\label{para-lem:2.1} Suppose $\Omega$ is a bounded domain in $\mathbb{R}^d$ and $a, c\in C^0(\bar{\Omega})$, $c\geq 0$, then the eigenvalue problem
\begin{align}\label{yy2}
\mathcal{L}\psi =\mu\,\psi~~ \text{with} ~~ \psi_{\partial\Omega}=0
\end{align}

has a countable set of positive eigenvalues $\mu_1\le\mu_2\le\cdots$,  with its corresponding eigenfunctions 
$\{\phi_k\}_{k=1}^\infty$ forming an orthogonal basis of $L^2(\Omega)$. 
Moreover, there exist constants $C_1,C_2>0$ such that 
$C_1 k^{2/d}\le \mu_k\le C_2k^{2/d}$ for all $k=1,2,\cdots.$
\end{prop} 

Using Proposition \ref{para-lem:2.1}, we can derive the spectral property of the forward operator $\mathcal{F}_T$. Specifically, we consider $\mathcal{F}_{T} f=u(\cdot, T)$, where $\mathcal{F}_T$ is the parabolic operator defined on $\Omega \subset R^{d}$ as follows:

\begin{align}\label{forwardoperator}
 \begin{cases}u_{t}+\mathcal{L} u=0 & \Omega \times(0, T) \\
u(x,t)=0 & \partial \Omega \times(0, T) \\
u(x, 0)=f(x) & \Omega.\end{cases} 
\end{align}
Specifically, we have the following lemma with respect to the eigenvalue distribution of the forward operator $\mathcal{F}_T$. 
\begin{lemma}\label{para-lem2.1-eigen}
The eigenvalue problem
\begin{align}\label{yy3}
(\psi,v)=\rho (\mathcal{F}_T\psi ,\mathcal{F}_Tv) \q \forall\, v\in L^2(\Omega)
\end{align}

has a countable set of positive eigenvalues $0<\rho_1\le\rho_2\le\cdots$ and the corresponding eigenfunctions are the same $\phi_k$ as in Proposition \ref{para-lem:2.1}. Moreover, $\rho_k=O(e^{2Tk^{2/d}})$ for all $k=1,2,\cdots$.
\end{lemma}

\begin{proof} Firstly, we consider the eigenvalue problem
\begin{align}\label{para-exameigenS2}
\psi=\eta\, \mathcal{F}_T\psi.
\end{align}
Let $\{\phi_k\}_{k=1}^\infty$ be eigenfunctions of the problem \eqref{yy2} that form an orthogonal basis of $L^2(\Omega)$. We write $f=\sum_{k=1}^\infty f_k \phi_k$ with a set of coefficients $f_k$. Let $u=\sum_{k=1}^\infty u_k(t) \phi_k$ be the solution of the problem \eqref{zz0}. By plugging both expressions of $f$ and $u$ into the first equation of \eqref{zz0} and noting the fact that $\mathcal{L}\phi_k=\mu_k\phi_k$, we can compare the coefficients of $\phi_k$  on both sides of the equation to obtain the following decoupled ordinary differential equations:
  
\begin{eqnarray*}
u'_k(t) + \mu_k u_k&=&0\ \ \ \ \mbox{in } ~(0, T)\,\\
u_k(0)&=&f_k.
\end{eqnarray*}
Solve this simple ordinary differential equation, $u_k(T)=e^{-\mu_kT}f_k$. Noting that $\mathcal{F}_Tf=u(\cdot, T)=\sum^\infty_{k=1} u_k(T) \phi_k$, we can formally write 
\begin{align}
 \mathcal{F}_T\Big(\sum_{k=1}^\infty f_k \phi_k\Big)=  \sum_{k=1}^\infty e^{-\mu_kT} f_k \phi_k.
\end{align}
Since $\{\phi_k\}_{k=1}^\infty$ is an orthogonal basis of $L^2(\Omega)$, we can readily see that the eigenvalue problem \eqref{para-exameigenS2} has a countable set of positive eigenvalues $\{\eta_k\}_{k=1}^\infty$, with $\{\phi_k\}_{k=1}^\infty$ being their corresponding eigenfunctions. By Proposition \ref{para-lem:2.1}, we have $\eta_k=e^{\mu_kT}=O(e^{Tk^{2/d}})$. Therefore, the eigenvalue problem \eqref{yy3} has a countable set of eigenvalues $\{\rho_k\}_{k=1}^\infty$ that satisfy $\rho_k=O(e^{2Tk^{2/d}})$. 
This completes the proof.  
\end{proof}  


We also need the following estimate regarding the covering entropy of Sobolev space \cite{Birman}, which plays an important role in this paper. For the completeness of this paper, we present the definition of the covering number. For a semi-metric space $(V,d)$ and any $\varepsilon>0$, the covering number $N(\varepsilon,BV,d)$ is the minimum number of $\varepsilon$-balls that cover the unit ball $BV$ of $V$ and $\log N(\varepsilon,V,d)$ is called the covering entropy which is an important quantity to characterize the complexity of the set $V$.

\begin{prop}\label{para-lem:2.3-entropy}
Let $Q$ be the unit cube in $\R^d$ ($d=1,2,3$) and $BH^{2}(Q)$ be the unit sphere of space 
$H^{2}(Q)$. Then it holds for sufficiently small $\varepsilon>0$ that 
\begin{align}
\log N(\varepsilon, BH^{2}(Q), \|\cdot\|_{C(Q)})\le C\varepsilon^{-d/2},
\end{align}
where $N(\varepsilon, BH^{2}(Q), \|\cdot\|_{C(Q)})$ means the covering number of Sobolev space $H^2(\Omega)$.
\end{prop}

With the above two lemmas, we will prove the following stochastic convergence of the inverse problem \eqref{para-ori} based on the framework developed in \cite{Chen-Zhang2022}. A random variable $X$ is sub-Gaussian with parameter $\sigma$ if it satisfies
\begin{align}\label{sub-gau-def}
{E}\left[e^{z(X-{E}[X])}\right]\le e^{\frac 12\sigma^2z^2}\ \ \ \ \forall z\in\R.
\end{align}

The probability distribution function of a sub-Gaussian random variable has an exponentially decaying tail, which is, if $X$ is a sub-Gaussian random variable, then
\begin{align}\label{exp-tail}
{P}(|X-{E} [X]|\geq z)\leq 2e^{-\frac 12 z^2/\sigma^2}\ \ \forall z>0. 
\end{align}

\begin{theorem}
    \label{thm:convergence} Let $f^*$ be the ground truth of observation $m_i=\mathcal{F}_Tf^*(x_i)+e_i$, $i=1, \cdots, n$, $f_{n}$ is the solution of the following problem: 
\begin{align}
    \min _{f \in L^{2}(\Omega)}\left\|\mathcal{F}_{T} f-m\right\|_{n}^{2}+\lambda\|f\|_{L^2(\Omega)}^{2}. \label{BoundEstimate0}
\end{align} 
For the first case, if $e_i$ are independent random variables with zero expectations and bounded variance $\sigma^2$, then,
\begin{align}
    E\left\|\mathcal{F}_{T} f_{n}-\mathcal{F}_{T} f^{*}\right\|_{n}^{2} \leqslant C \lambda\left\|f^{*}\right\|^{2}+C \frac{\sigma^{2}}{n \lambda^{d / 4}}, \label{BoundEstimate1}\\
    E\left\|f_{n}-f^{*}\right\|_{L^2(\Omega)}^{2} \leqslant C\left\|f^{*}\right\|_{L^2(\Omega)}^{2}+C \frac{\sigma^{2}}{n \lambda^{1+d/4}}.\label{BoundEstimate2}
\end{align}
 For the second case, if $e_i$ are sub-Gaussian random variables with parameter $\sigma$ and zero expectations, and choose $\lambda^{1 / 2+d / 8}=\mathcal{O}( \sigma n^{-1 / 2}\rho_0)$  with $\rho_0^{-1}=\left\|f^{*}\right\|_{L^2(\Omega)}+\sigma n^{-1/2}$, then,
\begin{align}
P\left(\left\|\mathcal{F}_{T} f_{n}-\mathcal{F}_{T} f^{*}\right\|_{n} \geqslant \lambda^{1 / 2} \rho_{0} z\right) \leq 2 e^{-C z^{2}} ,\label{BoundEstimate3}\\
P\left(\left\|f^{*}-f_{n}\right\|_{L^2(\Omega)} \geqslant \rho_{0} z\right) \leq 2 e^{-C z^{2}}.\label{BoundEstimate4}
\end{align}
\end{theorem}
\begin{proof}
For the first case, with Proposition \ref{para-lem2.1-eigen}, the eigenvalue distribution of \eqref{yy3} satisfies that $\eta_k\geq Ck^{4/d}$. The estimates \eqref{BoundEstimate1} and \eqref{BoundEstimate2} follow from Theorem 2.3 in \cite{Chen-Zhang2022}.

For the second case, with Proposition \ref{para-lem:2.3-entropy}, the entropy number of space $H^2(\Omega)$ has the order $\varepsilon^{-d/2}$ . The estimates \eqref{BoundEstimate3} and \eqref{BoundEstimate4} follow from Theorem 2.8 in \cite{Chen-Zhang2022}.
\end{proof}

Balance the two terms of the right hand side in \eqref{BoundEstimate1}-\eqref{BoundEstimate2}, the optimal choice $\lambda$ has the following form,
\begin{align}\label{optimal_lambda}
    \lambda_{optimal}^{1 / 2+d / 8}=\mathcal{O}( \sigma n^{-1 / 2}\left\|f^{*}\right\|_{L^2(\Omega)}^{-1}).
\end{align}

The theorem above only gives the stochastic convergence at the final time, while the convergence rate for the time $0<t<T$ should also be considered in BHCP. Next, we start to prove the main theorem of this paper to prove this convergence. Before proving this main theorem, we show two important estimates related to the forward operator $\mathcal{F}_{t}$.

\begin{lemma}\label{lemma:Kt}
Denote the operator $K_t=\mathcal{F}_{t}\left(\mathcal{F}_{T}^{*} \mathcal{F}_{T}+\lambda I\right)^{-1}$, 
where $\mathcal{F}_{t}$ is the forward operator defined in the equation \eqref{forwardoperator}. Then, we have the estimate $\left\|K_t\right\|^{2} \leqslant \frac{1}{\lambda^{2-t / T}}$.
\end{lemma}
\begin{proof}
For any $u \in L^{2}(\Omega)$, let $u=\sum_{i=1}^{\infty} u_{i} \phi_{i}$,
where $\phi_i$ are normalized eigenfunctions of $L$ in Proposition \ref{para-lem:2.1}. Moreover, we have $\|u\|_{L^2(\Omega)}^{2}=\sum_{i=1}^{\infty} u_{i}^{2}$. Simple calculations give us that
\begin{align}
\left\|\mathcal{F}_{t}\left(\mathcal{F}_{T}^{*} \mathcal{F}_{T}+\lambda I\right)^{-1} u\right\|^{2}&=\sum_{i=1}^{\infty} \frac{e^{-2 t \mu_{i}}}{(e^{-2T \mu_{i}}+\lambda)^2} u_{i}^{2} \nonumber\\
&=\sum_{i=1}^{\infty} \frac{e^{-2 t \mu_{i}}}{\left(e^{-2 T \mu_{i}}+\lambda\right)^{t / T}} \frac{1}{\left(e^{-2 T \mu_{i}}+\lambda\right)^{2-t / T}} u_{i}^{2} \nonumber\\
& \leqslant \sum_{i=1}^{\infty} \frac{1}{\left(e^{-2 T \mu_{i}}+\lambda\right)^{2-t / T}} u_{i}^{2} \nonumber\\
& \leqslant \frac{1}{\lambda^{2-t / T}} \sum_{i=1}^{\infty} u_{i}^{2}=\frac{1}{\lambda^{2-t / T}}\| u\|_{L^2(\Omega)}^{2}. 
\end{align}
This proves the lemma. 
\end{proof}
Using a similar approach, we can prove the following lemma. 
\begin{lemma}\label{lemma:STST}
Let $\mathcal{F}_{t}$ denote the forward operator defined in equation \eqref{forwardoperator}. Then, we have the estimate 
$\left\|\mathcal{F}_{T}\left(\mathcal{F}_{T}^{*} \mathcal{F}_{T}+\lambda I\right)^{-1} \mathcal{F}_{t}\right\|^{2} \leqslant \lambda^{t / T-1}$.
\end{lemma}
\begin{proof}
For any $u \in L^{2}(\Omega)$, let $u=\sum_{i=1}^{\infty} u_{i} \phi_{i}$,
where $\phi_i$ are normalized eigenfunctions of $\mathcal{L}$. Moreover, we have $\|u\|^{2}=\sum_{i=1}^{\infty} u_{i}^{2}$. Similar to the proof of Lemma \ref{lemma:Kt}, we just need to estimate the following term,
\begin{align}
 &   \sum_{i=1}^{\infty} \frac{e^{-2 t \mu_{i}} e^{-2 T \mu_{i}}}{\left(e^{-2 T \mu_{i}}+\lambda\right)^{2}} u_{i}^{2}\nonumber\\
=&\sum_{i=1}^{\infty} \frac{e^{-2 t \mu_{i}}}{\left(e^{-2 T \mu_{i}}+\lambda\right)^{t / T}} \frac{e^{-2 \tau \mu_{i}}}{e^{-2 \tau \mu_{i}}+\lambda} \frac{1}{\left(e^{-2 T \mu_{i}}+\lambda\right)^{1-t / T}} u_{i}^{2}\nonumber \\
\leq& \lambda^{t / T-1}\| u\|_{L^2(\Omega)}^{2}.
\end{align}
This proves the lemma.
\end{proof} 
The estimate \eqref{BoundEstimate1} only demonstrates the convergence of $E\left\|\mathcal{F}_{T} f_{n}-\mathcal{F}_{T} f^{*}\right\|_{n}^{2}$ for the final time $T$. To obtain the stochastic convergence result for any time $t\in [0,T]$, we need to estimate the error term $E\left\|\mathcal{F}_{t} f_{n}-\mathcal{F}_{t} f^{*}\right\|_{L^2(\Omega)}$ for $0\leq t \leq T$. Equipped with these two lemmas above, we can derive the following main theorem on the error of the backward heat conduction problem. 

To simplify the proof of the main theorem, we will choose $x_{i}, w_{i}$ in a way that the following property is satisfied:
$\sum_{i=1}^{n} w_{i}^{2} u^{2}\left(x_{i}\right)$ is approximation of $\|u\|^{2}_{L^{2}(\Omega)}$, $\left|\left\|u_{n}\right\|_n-\left\| u\right\|_{L^{2}(\Omega)}\right| \leqslant C d_{max}^2\| u \|_{H^{2}(\Omega)}$ with $d_{max}\sim n^{-\frac{1}{d}}$. 
This assumption assures that the discrete norm $\|\cdot\|_n$ is close to $L^2$ norm. This is not a very strong assumption. For example, let the operator $I_{h}$ be such that $\left\|I_{h} u\right\|_{L^2(\Omega)}^{2}=\sum_{i=1}^{n} w_{i}^{2} u^{2}\left(x_{i}\right)$. One could find many kinds of discrete inner products $(\cdot,\cdot)_n$ in the practice of numerical integration. For instance,  we could choose $x_{i}$'s as the nodes of a regular finite element mesh and $I_{h}$ as the finite element interpolation operator. Then $\left\|I_{h} u-u\right\|_{L^2(\Omega)} \leqslant C h^2\| u \|_{H^{2}(\Omega)}$, where $h \sim n^{-\frac{1}{d}}$.

\begin{theorem}[Main Theorem]\label{main-thm}
Assume the discrete semi-norm $\|\cdot\|_n$ is a good approximation of $L^2$ norm in the sense above. Let $f^*$ be the ground truth of the final time  observation $m_i=\mathcal{F}_Tf^*(x_i)+e_i$, $i=1, \cdots, n$, and $f_{n}$ be the solution of the following problem: 
\begin{align}
    \min _{f \in L^{2}(\Omega)}\left\|\mathcal{F}_{T} f-m\right\|_{n}^{2}+\lambda\|f\|_{L^2(\Omega)}^{2}. \label{BoundEstimate01}
\end{align} 
For the first case, if $e_i$ are independent random variables with zero expectations and bounded variance $\sigma^2$, then,
\begin{align}\label{eqn:time_t_estimate}
    E\left\|\mathcal{F}_{t} f_{n}-\mathcal{F}_{t} f^{*}\right\|_{L^2(\Omega)} \leqslant C \lambda^{t / 2 T}\left\|f^{*}\right\|_{L^2(\Omega)}+ C \frac{\sigma^{2}}{n} \frac{1}{\lambda^{d / 4}} \lambda^{t/T-1}.
\end{align}
If we take the optimal parameter \eqref{opti-lamda}  as $\lambda^{1/2+d / 8}=\mathcal{O}(\sigma n^{-1 / 2}\left\|f^{*}\right\|_{L^2(\Omega)}^{-1})$ (balancing the first and second part in the righthand side of \eqref{eqn:time_t_estimate}),
\begin{align}\label{eqn:time_t_estimate2}
E\left\|\mathcal{F}_{t} f_{n}-\mathcal{F}_{t} f^{*}\right\|_{L^2(\Omega)} \leq C \lambda^{t / 2 T}\| f^{*} \|_{L^2(\Omega)}. 
\end{align} 
For the second case, if $e_i$ are sub-Gaussian random variables with parameter $\sigma$ and zero expectations, choose $\lambda^{1 / 2+d / 8}=\mathcal{O}( \sigma n^{-1 / 2}\rho_0)$  with $\rho_0^{-1}=\left\|f^{*}\right\|_{L^2(\Omega)}+\sigma n^{-1/2}$, then,
\begin{align}\label{eqn:time_t_estimate3}
P\left(\left\|\mathcal{F}_{t} f_{n}-\mathcal{F}_{t} f^{*}\right\|_{n} \geqslant \lambda^{t / 2 T} \rho_{0} z\right) \leq 2 e^{-C z^{2}}.
\end{align}
\end{theorem} 

To make the proof of the main theorem clearer, we will present three lemmas that will be used in the proof. 

\begin{lemma}\label{lemma1}
Assuming the same conditions as stated in Theorem \ref{main-thm}, we can assert the existence of an operator $P_{T}^{*}: R^n \to L^2(\Omega)$ that acts as the adjoint operator with respect to the inner product $(\cdot, \cdot)_{n}$.
\end{lemma}
\begin{proof}
Let $P_{T}^{*}:R^n\to L^2(\Omega)$ be the adjoint operator with respect to the inner product $(\cdot, \cdot)_{n}$, i.e. 
\begin{align}
\left(u, \mathcal{F}_{T} v\right)_{n}=\left(P_{T}^{*} u, v\right)_{L^{2}(\Omega)}, \forall u \in R^{n}, v \in L^{2}(\Omega).
\end{align} 
Let $\mathcal{F}_{T}^{*}$ be the adjoint operator such that, 
\begin{align}
    \forall u, v \in L^{2}, \quad\left(u, \mathcal{F}_{T} v\right)_{L^{2}}=\left(\mathcal{F}_{T}^{*} u, v\right)_{L^{2}}.
\end{align}
Moreover we should point out that $\mathcal{F}_{T}$ is self-adjoint \cite{renardy2006introduction,Pazy:83}, meaning that$\mathcal{F}_{T}^{*}=\mathcal{F}_{T}$.

First, we will show the existence of the adjoint operator $P_{T}^{*}$. Fixing $u$, let $F_u(v)=\left(u, \mathcal{F}_{T} v\right)_n$, then 
$F_u$ is a bounded linear operator on $L^2\to R$:
\begin{align}
    F_u(v)\leq \|u\|_n\|\mathcal{F}_Tv\|_n\leq \|u\|_n\|v\|_{L^2}.
\end{align}
According to the Riesz representation Theorem, there exists a unique $f_u\in L^2$ such that $F_u(v)=\left(u, \mathcal{F}_{T} v\right)_n=(f_u,v)$, for all $v \in L^2$. We then define the linear operator $P_T^*u=f_u$ for any $u\in R^n$. 
\end{proof} 

\begin{lemma}\label{lemma2}
Under the same assumptions as stated in Theorem \ref{main-thm}, we can conclude that  
$$\|\left(P_{T}^{*} \mathcal{F}_{T}-\mathcal{F}_{T}^{*} \mathcal{F}_{T}\right) f\|_{L^{2}(\Omega)}  \leqslant Cd_{max}^2\left\|f\right\|_{L^{2}(\Omega)},$$
where $C$ is a constant.
\end{lemma} 
 
\begin{proof}
Denote $F=\left(P_{T}^{*} \mathcal{F}_{T}-\mathcal{F}_{T}^{*} \mathcal{F}_{T}\right) f$. We can calculate the 
$L^2$ norm of $F$ as follows:
\begin{align}
\|F\|_{L^2(\Omega)}^{2}=&(F, F)=\left(\left(P_{T}^{*} \mathcal{F}_{T}-\mathcal{F}_{T}^{*} \mathcal{F}_{T}\right)f, F\right)_{L^2(\Omega)} \nonumber\\
=&\left(\mathcal{F}_{T}f, \mathcal{F}_{T} F\right)_{n}-\left(\mathcal{F}_{T}f, \mathcal{F}_{T} F\right)_{L^{2}(\Omega)} \nonumber\\
=&\left(I_{h} \mathcal{F}_{T}f, I_{h} \mathcal{F}_{T} F\right)_{L^2(\Omega)}-\left(\mathcal{F}_{T}f, \mathcal{F}_{T} F\right)_{L^{2}(\Omega)} \nonumber\\
=&\left(I_{h} \mathcal{F}_{T}f-\mathcal{F}_{T}f, I_{h} \mathcal{F}_{T} F\right)_{L^{2}(\Omega)} \nonumber\\
&+\left(\mathcal{F}_{T}f, I_{h} \mathcal{F}_{T} F-\mathcal{F}_{T} F\right)_{L^{2}(\Omega)} \nonumber\\
\leqslant & C d_{max}^2\left\|f\right\|_{L^{2}(\Omega)} \| F\|_{L^{2}(\Omega)} 
\end{align}
Therefore, we obtain that 
\begin{align}
  \|F\|_{L^{2}(\Omega)} & \leqslant Cd_{max}^2\left\|f\right\|_{L^{2}(\Omega)},
 \end{align}
\end{proof}

\begin{lemma}\label{lemma3}
With the same assumption as Theorem \ref{main-thm}, we can conclude that 
\begin{equation}\label{boundIII}
    E\|K_t P_{T}^{*} e\|^2_{L^2(\Omega)}\leqslant C \frac{\sigma^{2}}{n} \frac{1}{\lambda^{d / 4}} \lambda^{t/T-1}.
\end{equation}
\end{lemma}

\begin{proof}
 For any $v\in {L^{2}(\Omega)}$, we know that 
\begin{align}
(K_t P_{T}^{*} e, v)_{L^2(\Omega)} &=\left(\mathcal{F}_{t}\left(\mathcal{F}_{T}^{*} \mathcal{F}_{T}+\lambda I\right)^{-1} P_{T}^{*} e, v\right)_{L^{2}(\Omega)}  \nonumber\\
&=\left(e, \mathcal{F}_{T}\left(\mathcal{F}_{T}^{*} \mathcal{F}_{T}+\lambda I\right)^{-1} \mathcal{F}_{t} v\right)_{n} \nonumber\\
&=\left(e, \mathcal{F}_{T} v_{1}\right)_{n},
\end{align}
where 
$v_{1} \stackrel{\Delta}{=}\left(\mathcal{F}_{T}^{*} \mathcal{F}_{T}+\lambda I\right)^{-1} \mathcal{F}_{t} v$. 

Assuming that $\beta_1\le\beta_2\le\cdots\le\beta_n$ are the eigenvalues of the problem
\begin{align}\label{x2}
(\psi,v)=\beta(\mathcal{F}_T\psi,\mathcal{F}_Tv)_n, \quad \forall v\in V_n, 
\end{align}
with corresponding eigenfunctions $\{\psi_k\}^n_{k=1}$, which form a set of orthonormal basis functions spanning $V_n=\m{span}\{\psi_1, \cdots, \psi_n\}$ under the inner product $(\mathcal{F}_T\cdot,\mathcal{F}_T\cdot)_n$, we can conclude that the eigenvalues satisfy $\beta_k\geq Ck^{4/d}$ by using Lemma \ref{para-lem2.1-eigen} in this paper and Lemma 2.2 in \cite{Chen-Zhang2022}. Therefore, we have
$(\mathcal{F}_T\psi_k,\mathcal{F}_T\psi_l)_n=\delta_{kl}$, and as a result, $(\psi_k,\psi_l)=\beta_k\delta_{kl}$, $k,l=1,2,\cdots,n$ and $\mathcal{F}_T\psi_i(x_j)=\delta_{ij}$. We can define the interpolation operator $I$ from $L^2$ to $V_n$ as:
\begin{align}
Iv=\sum_{i=1}^n (\mathcal{F}_Tv)(x_i) \psi_i.
\end{align}
This operator satisfies that $\|Iv\|\leq \|v\|$. For the complete proof, please refer to Lemma 2.2 in \cite{Chen-Zhang2022}. Moreover, we can define  $e'_i=nw_ie_i$, which means that $\{e'_i\}$ are also independent random variables and have the same order variance as $\{e_i\}$ since $w_i=O(1/n)$. It is easy to verify that $\mathcal{F}_T(v-Iv)(x_i)=0$ for all $x_i$.
   
Now for $v_1$, 
we have the expansion $v_1(x)=Iv_1+v_1-Iv_1=\sum^n_{k=1}u_k\psi_k(x) +v_1-Iv_1$, where $u_k=(\mathcal{F}_Tv,\mathcal{F}_T\psi_k)_n$ for $k=1,2,\cdots,n$.  Therefore, 
$\|\mathcal{F}_TIv_1\|_n^2+\lambda\|Iv_1\|^2_{L^2}=\sum^n_{k=1}(\lambda\beta_k+1)u_k^2$.
By the Cauchy-Schwarz inequality, we can easily obtain:  
\begin{align}
(e,\mathcal{F}_Tv_1)_n^2 &=  (e,\mathcal{F}_TIv_1)_n^2\\
 &= \frac{1}{n^2}\sum^n_{i=1}e'_i\left(\sum^n_{k=1} u_k \psi_k(x_i)\right) = \frac{1}{n^2}\sum^n_{k=1}u_k\left(\sum^n_{i=1} e'_i \psi_k(x_i)\right) \\
&\leq \frac{1}{n^2}\sum^n_{k=1}(1+\lambda_n\beta_k)u_k^2\cdot\sum^n_{k=1}(1+\lambda_n\beta_k)^{-1}\Big(\sum^n_{i=1}e'_i (\mathcal{F}\psi_k)(x_i)\Big)^2 \\
&= \left(\|\mathcal{F}_Tv_1\|_n^2+\lambda\|Iv_1\|^2_{L^2}\right)\frac{1}{n^2}\cdot\sum^n_{k=1}(1+\lambda_n\beta_k)^{-1}\Big(\sum^n_{i=1}e'_i (\mathcal{F}\psi_k)(x_i)\Big)^2 \\
&\leq C \left(\lambda^{t / T-1} \| v \|_{L^2(\Omega)}^{2}+\lambda \cdot \lambda^{t/T-2} \|v\|^{2}_{L^2(\Omega)}\right)\frac{1}{n^2}\cdot\sum^n_{k=1}(1+\lambda_n\beta_k)^{-1}\Big(\sum^n_{i=1}e'_i (\mathcal{F}\psi_k)(x_i)\Big)^2\\
&\leq C \lambda^{t / T-1}\|v\|^{2}_{L^2(\Omega)}\frac{1}{n^2}\cdot\sum^n_{k=1}(1+\lambda_n\beta_k)^{-1}\Big(\sum^n_{i=1}e'_i (\mathcal{F}\psi_k)(x_i)\Big)^2,
\end{align}
where in the second-to-last inequality, we have applied Lemma \ref{lemma:Kt} and Lemma \ref{lemma:STST}. 

Then we have, 
\begin{align}
\sup_{v\in L^2}\frac{(K_t P_{T}^{*} e, v)^2}{\|v\|^{2}_{L^2(\Omega)}}=\sup_{v\in L^2}\frac{(e,\mathcal{F}_Tv_1)_n^2}{\|v\|^{2}_{L^2(\Omega)}} &\leq\lambda^{t / T-1}\frac{1}{n^2}\cdot\sum^n_{k=1}(1+\lambda_n\beta_k)^{-1}\Big(\sum^n_{i=1}e_i (\mathcal{F}\psi_k)(x_i)\Big)^2.
\end{align}
Taking expectations on both sides, we get
\begin{align}
E\sup_{v\in L^2}\frac{(K_t P_{T}^{*} e, v)^2}{\|v\|^{2}_{L^2(\Omega)}} &\leq E\lambda^{t / T-1}\frac{1}{n^2}\cdot\sum^n_{k=1}(1+\lambda_n\beta_k)^{-1}\Big(\sum^n_{i=1}e_i (\mathcal{F}\psi_k)(x_i)\Big)^2\\
&=\lambda^{t / T-1}\frac{1}{n^2}\cdot\sum^n_{k=1}(1+\lambda_n\beta_k)^{-1} E\Big(\sum^n_{i=1}e_i (\mathcal{F}\psi_k)(x_i)\Big)^2\\
&\leq C\frac{\sigma^{2}}{n} \frac{1}{\lambda^{d / 4}}\lambda^{t / T-1},
\end{align}
where the fact that $\beta_k\geq Ck^{4/d}$ is used in the last inequality.

Thus,
\begin{equation}
    E\|K_t P_{T}^{*} e\|^2_{L^2(\Omega)}\leqslant C \frac{\sigma^{2}}{n} \frac{1}{\lambda^{d / 4}} \lambda^{t/T-1}.
\end{equation}
\end{proof} 

We will also recall a useful lemma of Van De Geer concerning stochastic convergence. In terms of the terminology of the stochastic convergence order, we denote a random variable $X=\mathcal{O}_{p}(z)$ if $X$ is a sub-Gaussian random variable with zero expectation and parameter $z$.
 \begin{prop}
\label{Geer-entropy}[\cite{Geer2000}, lemma 8.4]
 Suppose $G$ is a function space and $\log N\left(\varepsilon, B_G,\left\|\cdot\right\|_{n}\right) \leqslant C \varepsilon^{-\gamma}$, where $N$ denotes the local cover number of the unit ball $B_G$. Then, we have
$$
\sup_{g \in G} \frac{\left|(e, g)_{n}\right|}{\|g\|_{n}^{1-\gamma/ 2}\|g\|_{G}^{\gamma/2}}=\mathcal{O}_{p}\left(\sigma n^{-1 / 2}\right).
$$
\end{prop}  
 
\textbf{Proof of the main theorem.} 
We consider the least-squares regularized minimization problem \eqref{BoundEstimate01}. We know that $f_{n}$ satisfies the following  variational form: 
\begin{align}
\left(\mathcal{F}_{T} f_{n}-m, \mathcal{F}_{T} v\right)_{n}+\lambda\left(f_{n}, v\right)_{L^{2}(\Omega)}=0, \quad \forall v \in L^{2}(\Omega).
\end{align}

Using Lemma \ref{lemma1}, we have,
\begin{align}\left(P_{T}^{*}\left(\mathcal{F}_{T} f_{n}-m\right), v\right)_{L^2(\Omega)}+\lambda\left(f_{n}, v\right)_{L^2(\Omega)}=0, \quad \forall v \in L^{2}(\Omega).
\end{align}
This implies that
\begin{align}
 P_{T}^{*} \mathcal{F}_{T}f_{n}+\lambda f_{n}=P_{T}^{*} m.
\end{align}

Since $m_{i}=\mathcal{F}_{T} f^{*}\left(x_{i}\right)+e_{i}$,
\begin{align}
\left(P_{T}^{*} \mathcal{F}_{T}+\lambda I\right) f_{n}=P_{T}^{*} \mathcal{F}_{T} f^{*}+P_{T}^{*} e,
\end{align}
where $e=(e_1,...,e_n)^T$.

It is clear that $\left(\mathcal{F}_{T}^{*} \mathcal{F}_{T}+\lambda I\right) f^{*}=\mathcal{F}_{T}^{*} \mathcal{F}_{T} f^{*}+\lambda f^{*}$, along with the above equality, we can obtain 
\begin{align}
\left(\mathcal{F}_{T}^{*} \mathcal{F}_{T}+\lambda I\right)\left(f_{n}-f^{*}\right)=\left(P_{T}^{*} \mathcal{F}_{T}-\mathcal{F}_{T}^{*} \mathcal{F}_{T}\right)\left(f^{*}-f_{n}\right)-\lambda f^{*}+P_{T}^{*} e.
\end{align}
This gives,
\begin{align}
f_n-f^{*}=\left(\mathcal{F}_{T}^{*} \mathcal{F}_{T}+\lambda I\right)^{-1}\left(\left(P_{T}^{*} \mathcal{F}_{T}-\mathcal{F}_{T}^{*} \mathcal{F}_{T}\right)\left(f^{*}-f_{n}\right)-\lambda f^{*}+P_{T}^{*} e\right)
\end{align}

Hence,
\begin{align}
    \mathcal{F}_{t}\left(f_{n}-f^{*}\right)&= K_t\left(P_{T}^{*} \mathcal{F}_{T}-\mathcal{F}_{T}^{*} \mathcal{F}_{T}\right)\left(f^{*}-f_{n}\right) -K_t\left(\lambda f^{*}\right) +K_t\left(P_{T}^{*} e\right)\\&= (I)+(II)+(III),
\end{align}
where $K_t=\mathcal{F}_{t}\left(\mathcal{F}_{T}^{*} \mathcal{F}_{T}+\lambda I\right)^{-1}$, $I=K_t\left(P_{T}^{*} \mathcal{F}_{T}-\mathcal{F}_{T}^{*} \mathcal{F}_{T}\right)\left(f^{*}-f_{n}\right)$, $II=-K_t\left(\lambda f^{*}\right)$ and $III=K_t\left(P_{T}^{*} e\right)$.

First, we estimate the term $I$.  Denote $F=\left(P_{T}^{*} \mathcal{F}_{T}-\mathcal{F}_{T}^{*} \mathcal{F}_{T}\right)\left(f^{*}-f_{n}\right)$. Using Lemma \ref{lemma2}, we obtain that 
\begin{align}
  \|F\|_{L^{2}(\Omega)} & \leqslant Cd_{max}^2\left\|f^{*}-f_{n}\right\|_{L^{2}(\Omega)}.
 \end{align}
Applying Lemma \ref{lemma:Kt}, the estimate for the term $I$ is, 
\begin{align}\label{boundI}
 (I) \leqslant C \frac{1}{\lambda^{1-t / 2 T}} d^2_{\max }\left\|f^{*}-f_{n}\right\|_{L^{2}(\Omega)}  \leqslant C \lambda^{t / 2 T} \| f^{*}\|_{L^{2}(\Omega)},
\end{align}
where we use the condition that $ d_{max}^2 \leqslant C \lambda$.

Next, we estimate the term $II$. This part is straightforward since we have 
\begin{align}\label{boundII}
\left\|K_t\left(\lambda f^{*}\right)\right\|_{L^2(\Omega)}  \leqslant C \frac{1}{\lambda^{1-t / 2 T}} \cdot \lambda \| f^{*} \| _{L^{2}(\Omega)} = C \lambda^{t / 2 T}\left\|f^{*}\right\|_{L^2(\Omega)}.
\end{align} 

Finally, the estimation of the term $III$ comes directly from Lemma \ref{lemma3}. 


Combining \eqref{boundI}, \eqref{boundII} and \eqref{boundIII}, we can prove that 
\begin{align}
    E\left\|\mathcal{F}_{t} f_{n}-\mathcal{F}_{t} f^{*}\right\|_{L^2(\Omega)} \leqslant C \lambda^{t / 2 T}\left\|f^{*}\right\|_{L^2(\Omega)}+ C \frac{\sigma}{n^{1/2}} \frac{1}{\lambda^{d / 8}} \lambda^{t/2T-1/2}.
\end{align}
In addition, if we take $\lambda^{1/2+d / 8}=\sigma n^{-1 / 2}\left\|f^{*}\right\|_{L^2(\Omega)}^{-1}$,
\begin{align}
E\left\|\mathcal{F}_{t} f_{n}-\mathcal{F}_{t} f^{*}\right\|_{L^2(\Omega)} \leq C \lambda^{t / 2 T}\| f^{*} \|_{L^2(\Omega)}. 
\end{align} 

Then \eqref{eqn:time_t_estimate} and \eqref{eqn:time_t_estimate2} are proved.

Next, we will prove the convergence in probabilities \eqref{eqn:time_t_estimate3} for the second case. As parts I and II only consist of deterministic terms, we just need to estimate the convergence in probability of the term III. To achieve this,  we will apply the theory of empirical process as described in \cite{Geer2000}, see Proposition \ref{Geer-entropy}.

In Proposition \ref{Geer-entropy}, let $G=H^2(\Omega)$, 
According to Proposition \ref{para-lem:2.3-entropy}, we have $\log N\left(\varepsilon, B_G,\left\|\cdot\right\|_{n}\right) \leqslant C \varepsilon^{-d/2}$.  For any $ v \in L^2(\Omega)$, let $g\triangleq \mathcal{F}_{T}\left(\mathcal{F}_{T}^{*} \mathcal{F}_{T}+\lambda I\right)^{-1} \mathcal{F}_{t} v \in H^{2}(\Omega)$. We can apply Proposition \ref{Geer-entropy} to obtain 
\begin{align}
\sup_{g} \frac{\left|(e, g)_{n}\right|}{\|g\|_{n}^{1-\frac{d}{4}}\|g\|_{H^2}^{d/4}}=\mathcal{O}_{p}\left(\sigma n^{-1 / 2}\right).
\end{align}
Then we apply Lemma \ref{lemma:Kt} and Lemma \ref{lemma:STST}, we have, 
\begin{align}
 (e, g)_{n}&=\mathcal{O}_{p}\left(\sigma n^{-1 / 2}\left\|g\right\|^{1-\frac{d}{4}}_n \|g\|^{d / 4}_{H^{2}} \right)\nonumber\\
&=\mathcal{O}_{p}\left(\sigma n^{-1 / 2}\left(\lambda^{t / 2 T-1 / 2}\right)^{1-\frac{d}{4}}\left(\lambda^{t / 2 T-1}\right)^{d / 4} \| v\|_{ L^2(\Omega)}\right)\nonumber\\
&=\mathcal{O}_{p}\left(\sigma n^{-1 / 2} \quad \lambda^{t / 2 T} \quad \lambda^{-\left(\frac{1}{2}+\frac{d}{8}\right)} \| v\|_{L^2(\Omega)}\right).
\end{align}
Since the optimal parameter $\lambda^{\frac{1}{2}+\frac{d}{8}}=\mathcal{O}\left(\sigma_{n}^{-1 / 2}\left\|f^{*}\right\|^{-1}_{L^{2}(\Omega)}\right)$, then, 
\begin{align}\label{term3}
(e,g)_n=\mathcal{O}_{p} \left(\lambda^{t / 2 T}\left\|f^{*}\right\|_{L^{2}(\Omega)}\|v\|_{L^2(\Omega)}\right).
\end{align}
For the term III, 
\begin{align}
(K_t P_{T}^{*} e, v)_{L^2(\Omega)} =\left(e, \mathcal{F}_{T} v_{1}\right)_{n} =\left(e, g\right)_{n}.
\end{align}
Apply the estimate \eqref{term3},
\begin{align}
(K_t P_{T}^{*} e, v)_{L^2(\Omega)}=\mathcal{O}_{p} \big(\lambda^{t / 2 T}\left\|f^{*}\right\|_{L^2(\Omega)} \| v \|_{L^2(\Omega)}\big).
\end{align}
Since $v\in L^2(\Omega)$ is an arbitrary function, then
\begin{align}
\left\|K_t P_{T}^{*} e\right\|_{L^2(\Omega)}=\mathcal{O}_{p} \big(\lambda^{t / 2 T}\| f^{*} \|_{{L}^{2}(\Omega)}\big).
\end{align}
This will give the estimate \eqref{eqn:time_t_estimate3}.

\section{Methodology}
In this section, we introduce a numerical method for solving the backward heat conduction problems \eqref{zz0} with noisy final time measurement. Our method consists of two parts. First, we will present a numerical method to solve the regularization problem \eqref{para-ori} for a given regularization parameter $\lambda$. Then, we will propose an iterative algorithm to determine the optimal $\lambda$, inspired by the convergence analysis in Section \ref{sec:analysis}.

\subsection{Iterative method for the inverse problem}
 We first define the functional $\mathcal{J}$ as follows:
\begin{align}
	\mathcal{J}[f]=\|\mathcal{F}_Tf-m\|^2_n+\lambda_n \|f\|_{L^2(\Omega)}^2.
\end{align}	 
Then we just need to solve the following misfit functional:
\begin{align}
	\mathop {\rm min}\limits_{f\in L^2(\Omega)}  \mathcal{J}[f].
\end{align}
\begin{lemma}\label{Fdifferentiable}
	The misfit functional $\mathcal{J}[f]$ is Fr$\acute{e}$chet-differentiable. 
\end{lemma}
\begin{proof} 
	From the definition of Fr$\acute{e}$chet differentiability, we need to compute
	\begin{align}\label{Frechet-differentiable}
		d\mathcal{J}[f](v)&=\lim_{t\rightarrow 0} \frac{\mathcal{J}[f+tv]-\mathcal{J}[f]}{t}\nonumber\\
		&= (\mathcal{F}_Tf-m,\mathcal{F}_Tv)_n + \lambda_n(f,v)\nonumber\\
		&= \big(\mathcal{F}_T^*(\mathcal{F}_Tf-m),v\big) + \lambda_n(f,v)\nonumber\\
		&= \big(\mathcal{F}_T^*(\mathcal{F}_Tf-m)+ \lambda_n f,v\big), 
	\end{align}
	where $v\in L^2(\Omega)$. In \eqref{Frechet-differentiable}, the second equality is easy to get from the quadratic form of the misfit functional $\mathcal{J}[f]$. Thus, one can directly obtain that 
	\begin{align}
		d\mathcal{J}[f]=\mathcal{F}_T^*(\mathcal{F}_Tf-m)+ \lambda_n f. \label{dJf}
	\end{align}
\end{proof}
The formula \eqref{dJf} in Lemma \ref{Fdifferentiable} allows us to apply the gradient descent method to minimize the discrepancy functional $\mathcal{J}[f]$. Let $f_{0}$ be an initial guess and $f_k$ denote the solution of the least-squares regularized minimization problem \eqref{para-ori} at the $k$-th iteration step. We update the iterative solution by 
\begin{align}
	f_{k+1}=f_{k}-\beta d\mathcal{J}[f_{k}], \quad \forall k\in \mathbb{N},
\end{align}
where ${\beta}$ is the step size.

We employ the backward Euler scheme to discretize the heat equation by the linear finite element method. Given a fully discrete scheme with linear finite element space $V_{fem}$, we define the standard approximate forward operator $\mathcal{F}_{fem}$: $L^2(\Omega)\rightarrow V_{fem}$. This operator, given any initial function $f$, provides the numerical solution at the final time, i.e., $\mathcal{F}_{fem}f$ gives the numerical solution at the final time for a given initial function $f$.  In such a discrete setting,  we turn to solve the following misfit functional:
\begin{align}
	\mathop {\rm min}\limits_{f\in V_{fem}}  \mathcal{J}_{fem}[f],
\end{align}
where the functional $\mathcal{J}_{fem}[f]$ has the form $\mathcal{J}_{fem}[f]=\|\mathcal{F}_{fem}f-m\|^2_n+\lambda_n \|f\|_{L^2(\Omega)}^2$. We compute the 
Fr$\acute{e}$chet derivative of $\mathcal{J}_{fem}[f]$ in the same spirit of \eqref{dJf} and obtain the following iterative scheme:
\begin{align}\label{eqn:back_iteration}
	f_{k+1}=f_{k}-\beta d\mathcal{J}_{fem}[f_{k}], \quad \forall k\in \mathbb{N},
\end{align}
where $\beta$ is the step size, $d\mathcal{J}_{fem}[f]=\mathcal{F}_{fem}^*(\mathcal{F}_{fem}f-m)+ \lambda_n f$, $\mathcal{F}_{fem}^*$ is the discrete operator for the $\mathcal{F}^*$ and $f_{0}$ is an initial guess.

 \subsection{An a posterior estimate of the regularization parameter $\lambda$}
Theorem \ref{main-thm}  and Theorem \ref{thm:convergence} indicate that the optimal choice of regularization parameter has the form:
\begin{align}\label{opti-lamda}
    \lambda^{1 / 2+d / 8}=\mathcal{O}( \sigma n^{-1 / 2}\left\|f^{*}\right\|_{L^2(\Omega)}^{-1}).
\end{align}
The corresponding error will be 
\begin{align}\label{opti-error}
    \|\mathcal{F}_{T} f_n-\mathcal{F}_{T} 
 f^*\|_{L^2(\Omega)}=\mathcal{O}(\lambda^{1/2} \|f^{*}\|_{L^2(\Omega)}).
\end{align} 
This estimate provides a rough order of the optimal parameter $\lambda$. Determining the optimal value of $\lambda$ can be challenging, and the optimal $\lambda$ may vary as $T$ changes. This variation can also be observed through numerical experiments. Obtaining accurate estimations of all the constants in this paper is necessary to determine the optimal parameter, but it is not an easy task. In this section, we present an alternative method for selecting the optimal value of $\lambda$, which varies as the final time $T$ changes.

This optimal parameter $\lambda$ depends on $T$ since $\|\mathcal{F}_Tf^*\|_{H^2(\Omega)}\leq C_T\|f^*\|_{L^2(\Omega)}$, where $C_T$ depends on $T$. Let $u^*=\mathcal{F}_{T} f^*\in H^2(\Omega)$ and we can approximately estimate this constant $C_T$ by $\frac{\|u^*\|_{H^2(\Omega)}}{\|f_i\|_{L^2(\Omega)}}$. Counting this constant into the proofs of Theorems \ref{thm:convergence} and \ref{main-thm}, we obtain that the optimal regularization parameter $\lambda$ has the following form:
\begin{align}\label{opti-lamda-final}
    \lambda^{1 / 2+d / 8}=\mathcal{O}( \frac{\|u^*\|^{d/4}_{H^2(\Omega)}}{\|f^{*}\|^{d/4}_{L^2(\Omega)}}\sigma n^{-1 / 2}\left\|f^{*}\right\|_{L^2(\Omega)}^{-1}),
\end{align}




In this way, we establish an appropriate method for estimating the correct time-dependent constant in \eqref{opti-lamda}. We can see from the numerical examples that this estimate is a good choice. This approach offers an a prior estimate of the optimal regularization parameter $\lambda$ with knowledge of the noise level $\sigma$ and true solution $f^*$. In many industrial applications, these prior parameters are exactly what one wants to know, and many methods assume the knowledge of this a prior information. In the following section, we will present an algorithm that determines the regularization parameter $\lambda$ in an a posteriori manner, without requiring prior knowledge of the noise level $\sigma$ or the true solution $f^*$.

In practical settings, since the initial value $f^*$ and noise level $\sigma$ are unknown, we cannot directly apply the suggested formula to find the optimal $\lambda$ from Eq.\eqref{opti-lamda-final}.  To address this issue, we propose a self-adaptive method inspired by fixed-point iteration to determine the optimal $\lambda$. Specifically, we begin with an initial guess, such as $\lambda_{ 0}=1$, and solve for $f_0$. After solving $f_i$ at the $i$-th step, we update $\lambda$ as follows:
\begin{align}\label{eqn:opt_lam_alg}
	\lambda_{ i+1}^{1 / 2+d / 8}=\frac{\|u^*\|^{d/4}_{H^2(\Omega)}}{\|f_i\|^{d/4}_{L^2(\Omega)}}n^{-1 / 2}\left\|\mathcal{F}_{T} f_{i}-m\right\|_{n}\left\|f_{i}\right\|_{L^{2}(\Omega)}^{-1}.
\end{align}
In the above formula, we estimate $\|u^*\|_{H^2(\Omega)}$ solely from the observation data $m$ using the denoising algorithm proposed in \cite{Chen-Zhang}. We will terminate the iteration of $\lambda$ once $\left\|f_{i}\right\|_{L^{2}(\Omega)}$ converges. The iteration details are presented in the following diagram, and the numerical performance of this iteration will be discussed in the following section.

\begin{algorithm}[h]
	\SetAlgoLined
	\textbf{Input}: Observation data $m$,  number of observation data $n$, error threshold of $\lambda$ $tol_\lambda$.\\
    Estimate $\|u^*\|_{H^2(\Omega)}$ from the data $m$ by the denoising algorithm proposed in \cite{Chen-Zhang}, denote this estimation by $H_{est}$;\\
	Setup initial guess $\lambda_{0}\leftarrow n^{-4 /(d+4)}$;\\
	Solve $f_0$ from $m$ with parameter $\lambda=\lambda_{0}$;\\
	$\lambda_{1}\leftarrow (\frac{H^{d/4}_{est}}{\|f_0\|^{d/4}_{L^2(\Omega)}}n^{-1 / 2}\left\|\mathcal{F}_{T} f_{0}-m\right\|_{n}\left\|f_{0}\right\|_{L^{2}(\Omega)}^{-1})^{\frac{1}{1/2+d/8}}$;\\
	$j\leftarrow 1$;\\
	\While{$|\lambda_{j}-\lambda_{j-1}|>tol_\lambda$}{
	
		Solve $f_{j}$ from $m$ with parameter $\lambda=\lambda_{j}$;\\
		$\lambda_{j}\leftarrow (\frac{H^{d/4}_{est}}{\|f_{j}\|^{d/4}_{L^2(\Omega)}}n^{-1 / 2}\left\|\mathcal{F}_{T} f_{j}-m\right\|_{n}\left\|f_{j}\right\|_{L^{2}(\Omega)}^{-1})^{\frac{1}{1/2+d/8}}$;\\
		$j\leftarrow j+1$;\\
	}
	
	\textbf{Output}: parameter $\lambda_{j}$, computed the initial function $f_{j}$.
	\caption{\textbf{An iterative algorithm of finding optimal parameter $\lambda$ and recovering initial value .}}
	\label{alg_lambda}
\end{algorithm}

\section{Numerical examples} 

In this section, we demonstrate the effectiveness of our proposed regularization method for solving BHCPs through numerical examples. Since our estimate only assumes $L_2$ initial data, we explore two types of initial data: a smooth function (such as a product of two $\sin$ functions) and a discontinuous bounded function (such as the indicator function of an $A$ shaped domain). We will start with the case of smooth initial data. First, we examine the stochastic convergence of our method, as outlined in our main theorem. Next, we demonstrate that our estimate provides a near-optimal choice of the regularization parameter. Additionally, we propose a scheme that uses fixed-point integration to determine optimal regularization parameters, even without \textit{a-priori} knowledge of the initial data. Finally, we apply our algorithm to invert discontinuous initial data.


\subsection{Configurations and performance in recovering $L_2$ initial data}
 For the first example, we consider a smooth initial condition on the $\Omega=[0,\pi]^2$ domain, i.e.,
\begin{align}\label{eq:smooth}
    f(x,y)=\sin(2x)\sin(2y) \quad (x,y)\in [0,\pi]^2,
\end{align}
see Fig.\ref{fig:smooth}(c).
To generate a synthetic observation, we first solve Eq.\eqref{zz0} for $t=[0,0.1]$ by finite element methods with $h=\frac{\pi}{16}$ in space discretizaion and $\delta t=1\times 10^{-3}$ in time. Then we make observations of $u(\cdot,t=0.1)$ at $n=20\times20$ equiv-distance distributed points $\{x_i\}_{i=1}^N$ over $\Omega$. At the $i$-th observation point, $x_i$, the observation is assumed to deviate from the ground truth $u(x_i,T)$ by an independent noise, i.e.,
\begin{align}\label{eq:obs_noise}
    m_i=u(x_i,T)+\sigma N_i,
\end{align}
where $\{N_i\}$ are i.i.d. standard normal distribution and $\sigma=0.05>0$ represents the standard derivation of observation error at each point.
\medskip
 
Then we present the performance of the regularization method in recovering the initial values. Fig.\ref{fig:smooth}(a) shows the noisy observation of $u$ at time $T=0.1$. Since the selection in $\lambda$ is crucial in the inverse initial value problem, we utilize the fixed point iteration proposed in Alg.\ref{alg_lambda} to find the optimal regularization. The detailed performance of the iterative algorithm \ref{alg_lambda} will be discussed in the second part of Sec.\ref{sec:optlamb}.  The result of inversion with
final step $\lambda=3.936\times 10^{-4}$ is presented in Fig. \ref{fig:smooth}(b). Despite the noisy observation, we find that our inversion is smooth and successfully recovers the ground truth in Fig.\ref{fig:smooth}(c).  
	 \begin{figure}[htbp]
	 	\centering
	 	\begin{subfigure}{0.32\textwidth}
	 		\includegraphics[width =\linewidth]{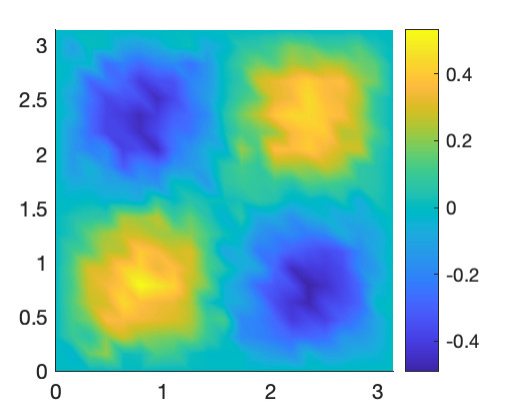}
	 		\caption{observation at $t=0.1$}
	 	\end{subfigure}
	 	\begin{subfigure}{0.32\textwidth}
	\includegraphics[width =\linewidth]{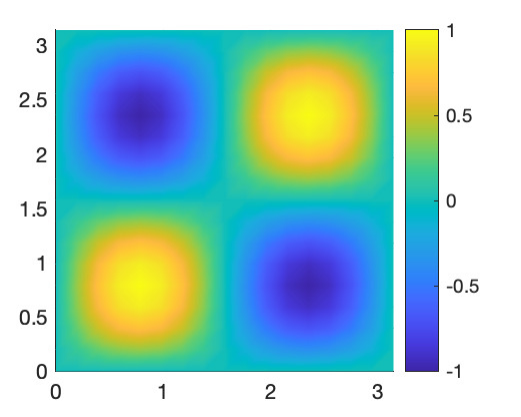}
	\caption{regularized inversion}
\end{subfigure}
	 	\begin{subfigure}{0.32\textwidth}
	\includegraphics[width =\linewidth]{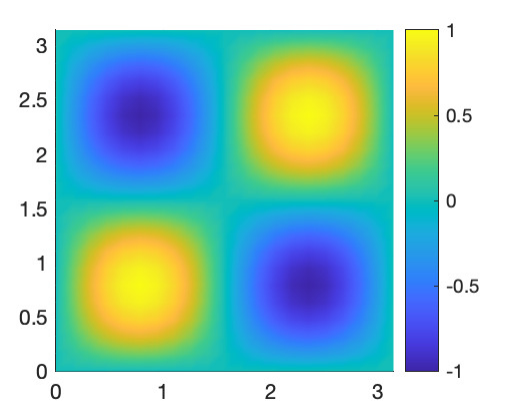}
	\caption{ground truth}
\end{subfigure}
	 	\caption{Performance of regularization algorithm when recovering smooth $L_2$ initial value } 
	 	\centering
	 	\label{fig:smooth}
	 \end{figure}

\paragraph{Verification of convergence}
Following up, we verify the convergence result in Sec.\ref{sec:analysis}.  

The interior time estimate in Eq.\eqref{eqn:time_t_estimate} is validated in Fig.\ref{fig:convergence}(a). The y-axis of the plot is in log scale, and the near-linear line indicates that the error exponentially decays as $t$ approaches the terminal time of $0.1$. This can be attributed to the fact that error grows exponentially fast when propagating back to the initial values.

To investigate the stochastic convergence in Theorem \ref{thm:convergence}, we generate $10000$ independent realizations of noisy observations in \eqref{eq:obs_noise} and record the errors. In Fig.\ref{fig:convergence}(b), we compare the errors with a standard normal distribution by using a quantile-quantile plot (Q-Q plot). The Q-Q plot is a statistical method for comparing two probability distributions by plotting their quantiles against each other. We can see a straight line in Fig.\ref{fig:convergence}(b), which indicates that the distribution of errors is similar to a normal distribution. This verifies the Gaussian tail estimate in Eq.\eqref{eqn:time_t_estimate3}.

\begin{figure}[htbp]
    \centering
    \begin{subfigure}{0.45\textwidth}
    \includegraphics[width=\linewidth]{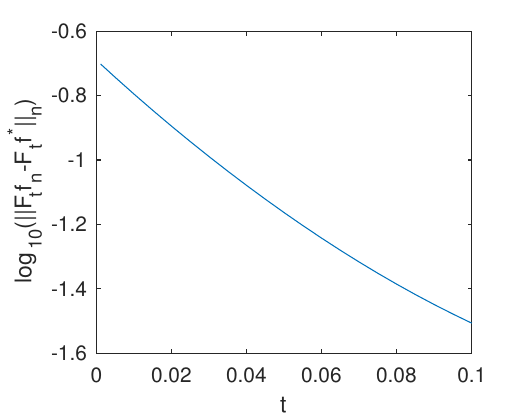}
    \caption{$||\mathcal{F}_t f_* - \mathcal{F}_t f||_{n}$ in one realization, y-axis is in log scale}
    \end{subfigure}~
    \begin{subfigure}{0.45\textwidth}
    \includegraphics[width=\linewidth]{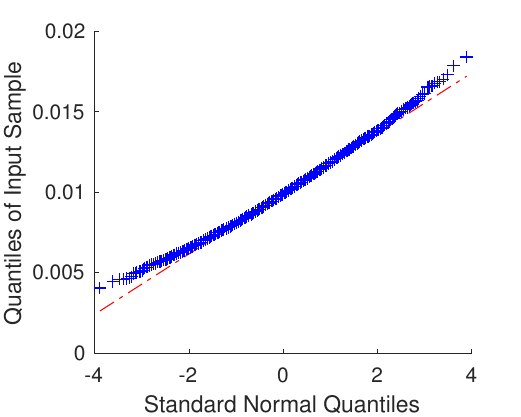}
    \caption{Q-Q plot of $||\mathcal{F}_Tf_* - \mathcal{F}_Tf||_{n}$}
    \end{subfigure}
    \caption{Convergence verification when recovering smooth $L_2$ initial value}
    \label{fig:convergence}
\end{figure}

\subsection{Optimality of the regularization parameter $\lambda$}\label{sec:optlamb}
One of the main contributions in our analysis in Sec.\ref{sec:analysis} is to provide a scale estimate of the optimal regularization parameter for the BHCPs. This estimate is valuable for practitioners as it helps them choose an appropriate value of the regularization parameter for their specific problems.

\paragraph{Verification of scale estimation}

To begin, we recall the estimation for the optimal $\lambda$ in \eqref{optimal_lambda}. Specifically, when we take $d=2$, we have the following: 
\begin{align}\label{optima_lambda2d}
    \lambda=\mathcal{O}( \sigma^{4/3} n^{-2 / 3}\left\|f^{*}\right\|_{L^2(\Omega)}^{-4/3}). 
\end{align}
In Fig.\ref{fig:optimal_lambda}, we demonstrate the application of various configurations for recovering smooth initial data. 
Specifically, we begin by setting $\sigma=10^{-1}$ and $n=10^4$. To find the optimal value of $\lambda$ in practice, we apply our inverse algorithm (excluding the fixed point interactive algorithm in Alg.\ref{alg_lambda}) with $\lambda$ ranging from $10^{-5}$ to $10^{-2}$. We then determine the corresponding $\lambda_*$ that minimizes the $L_2$ error between the inversion $f_n$ and the ground truth $f^*$.

In Fig.\ref{fig:optimal_lambda}(a), we keep the observation number $n$ and  ground truth $f^*$ unchanged while testing for varying values of $\sigma$ between $2.5\times10^{-2}$ and $0.8$, and determine the corresponding $\lambda_*$. The fitted slope is $1.12$, which is close to the analytical suggestion $\frac{4}{3}$ in Eq.\eqref{optima_lambda2d}. Similarly, in Fig.\ref{fig:optimal_lambda}(b) and (c), we test for different observation numbers $n$ and $L_2$ norm of the initial value $||f||_{L_2}$ while keeping the remaining configuration constant. In both cases, the practical best $\lambda$ adheres to the scale suggested analytically in Eq.\eqref{optima_lambda2d}.

\begin{figure}[htbp]
    \centering
     \begin{subfigure}{0.31\textwidth}
    \includegraphics[width=\linewidth]{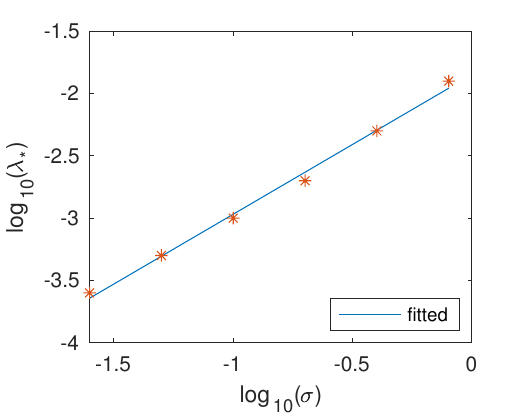}
    \caption{$\sigma$, fitted slope $1.12$}
    \end{subfigure}
   \begin{subfigure}{0.31\textwidth}
    \includegraphics[width=\linewidth]{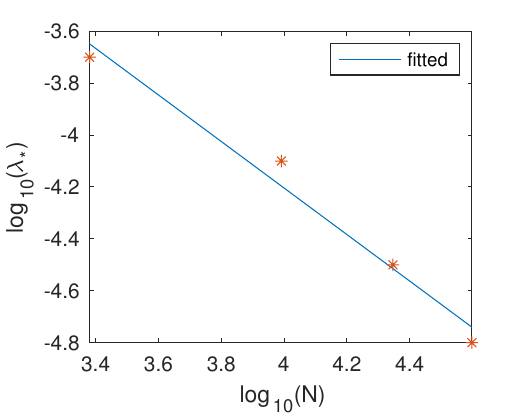}
    \caption{$n$, fitted slope $-0.8963$}
    \end{subfigure}
    \begin{subfigure}{0.31\textwidth}
    \includegraphics[width=\linewidth]{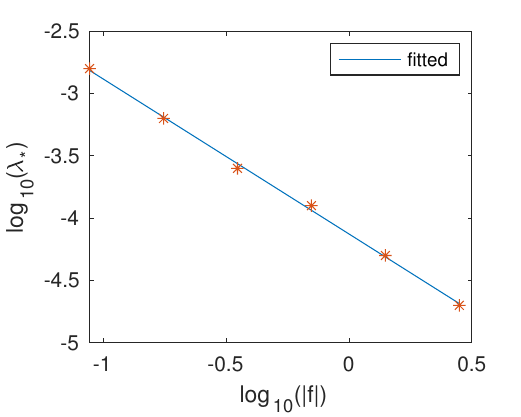}
    \caption{$||f^*||_{L_2}$, fitted slope $-1.2434$}
    \end{subfigure}
    \caption{Validation of optimal scale estimate of $\lambda$ \eqref{optimal_lambda} }
    \label{fig:optimal_lambda}
\end{figure}

\paragraph{Self adaptive method in finding optimal $\lambda$.}  In Fig.\ref{fig:adaptive_lambda}(a), we demonstrate that our iterative algorithm monotonically decreases the $l_2$ distance between recovered final time solution $Sf_i$ and the observation in $9$ steps. In Fig.\ref{fig:adaptive_lambda}(c),(d),(e), (f), we present the inversion result at $1$, $2$ $4$, and $6$ steps, respectively. Fig.\ref{fig:adaptive_lambda}(b) shows the $l_2$ distance to the observation at the $i$-th step against the difference between $\lambda_i$ and optimal $\lambda$. By iterating for $\lambda$, we can see that our algorithm gradually recovers the initial data $f$ in \eqref{eq:smooth}.
 
\begin{figure}[htbp]
    \centering
     \begin{subfigure}{0.49\textwidth}
    \includegraphics[width=\linewidth]{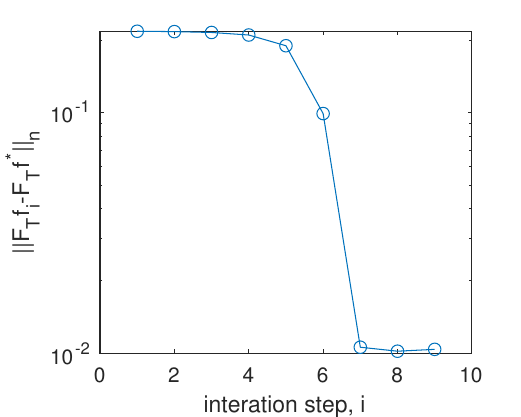}
    \caption{$l_2$ distance in each iteration}
    \end{subfigure}
    \begin{subfigure}{0.49\textwidth}
    \includegraphics[width=\linewidth]{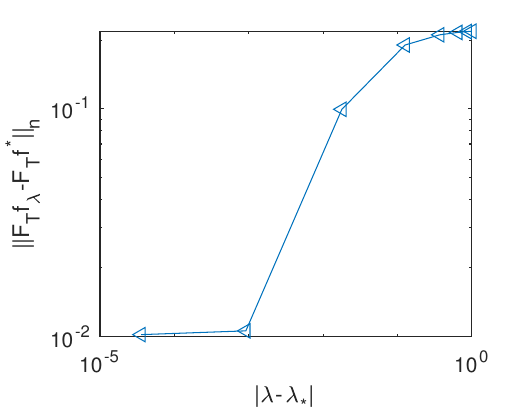}
    \caption{$l_2$ distance  vs $\lambda$ derivation}
    \end{subfigure}\\
    
   \begin{subfigure}{0.49\textwidth}
    \includegraphics[width=\linewidth]{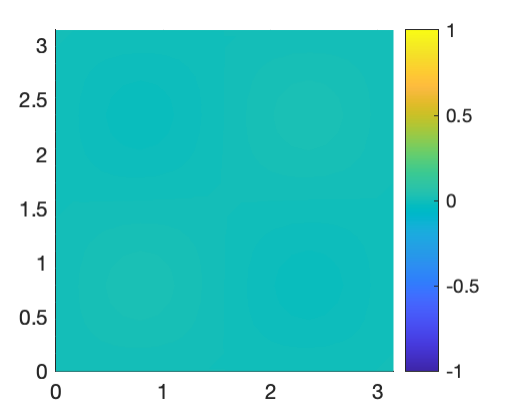}
    \caption{Inversion with $\lambda_0=1$}
    \end{subfigure}
    \begin{subfigure}{0.49\textwidth}
    \includegraphics[width=\linewidth]{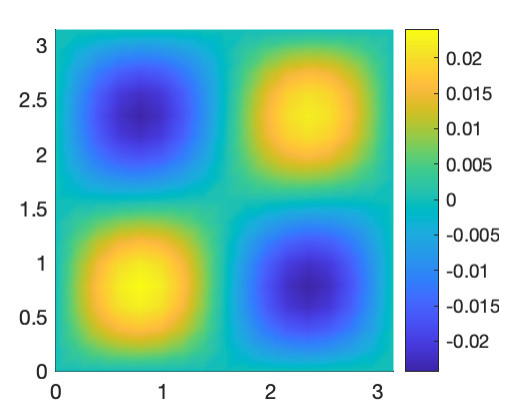}
    \caption{Inversion with $\lambda_2=0.6836$}
    \end{subfigure}\\
    
    \begin{subfigure}{0.49\textwidth}
    \includegraphics[width=\linewidth]{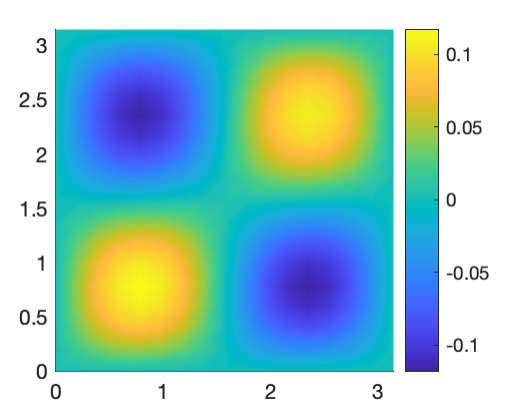}
    \caption{Inversion with $\lambda_4=0.1634$}
    \end{subfigure}
    \begin{subfigure}{0.49\textwidth}
    \includegraphics[width=\linewidth]{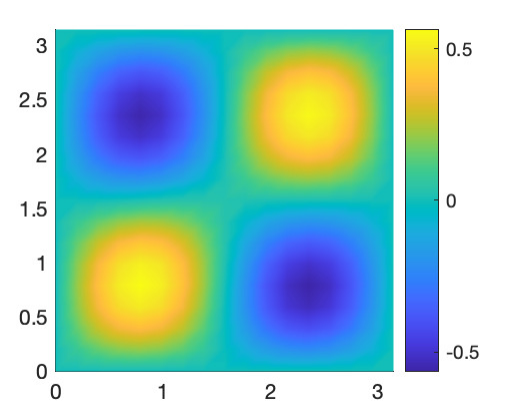}
    \caption{Inversion with $\lambda_6=0.0183$}
    \end{subfigure}
\caption{Self-adaptive method in finding the optimal $\lambda$, final result at step $9$ shown in Fig.\ref{fig:smooth}(c)}
    \label{fig:adaptive_lambda}
\end{figure}

\subsection{Performance in recovering discontinuous initial data}
In the second example, we consider the initial condition $f$ as the indicator function of an \textbf{A}-shape sub-domain, as shown in Fig.\ref{fig:A} (b). Compared to smooth initial data, the sharp edges of $f$ lead to a slower decay of the spectrum. In the forward problem, the higher frequency data decays rapidly against $T$, making it difficult to track. Thus, in this example, the standard derivation of the noise $\sigma$ is set to $10^{-4}$, and the final observation is at $T=0.05$ with $n=50\times 50$ equally spaced points. Fig.\ref{fig:A} (a) illustrates the generated observations. The remaining configuration is identical to the one used for smooth data presented in Fig.\ref{fig:smooth}.

In Fig.\ref{fig:A}(d-f), we present the step-wise regularized inversion result. Our iterative algorithm takes $2$ steps to achieve a near-optimal $\lambda=1.2644\times 10^{-6}$. The observation is two-fold. Firstly, we can see that the inversion cannot recover the sharp edge of $A$ but can accurately identify the shape of the sub-domain. This is due to the ill-posed nature of the problem. However, in contrast to the inversion result with an incorrect $\lambda$, as depicted in Fig.\ref{fig:A}(d-e), the iterative method Alg.\ref{alg_lambda}, based on our theoretical analysis in Sec.\ref{sec:analysis}, plays a crucial role in recovering $L_2$ non-smooth data.

In addition, we validate the asymptotic estimate against $t$ in \eqref{eqn:time_t_estimate2}. After applying the iterative algorithm to find the initial value $f_n$ and optimal $\lambda_*$, we apply the forward solver $\mathcal{F}_t$ for various $t$ and compare $\mathcal{F}_tf_n$ with $\mathcal{F}_tf^*$ at different $t$ values. We then determine the fitted slope of $\log_{10}(\|\mathcal{F}_tf_n-\mathcal{F}_tf^*\|)$ against $t$ to be $-131.02$, which is close to the predicted slope of $\frac{\log_{10}\lambda_*}{2T}=-135.81$.

	 \begin{figure}[htbp]
	 	\centering
	 	\begin{subfigure}{0.32\textwidth}
	 		\includegraphics[width =\linewidth]{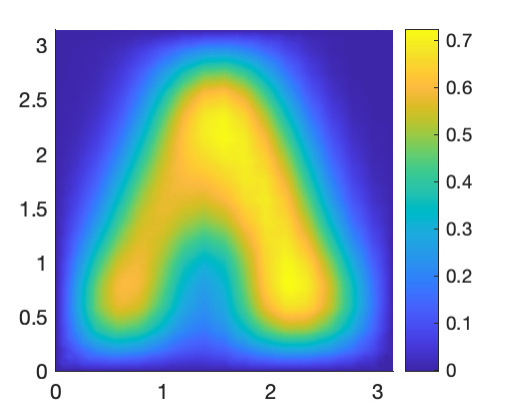}
	 		\caption{observation at terminal time}
	 	\end{subfigure}
	 	\begin{subfigure}{0.32\textwidth}
	\includegraphics[width =\linewidth]{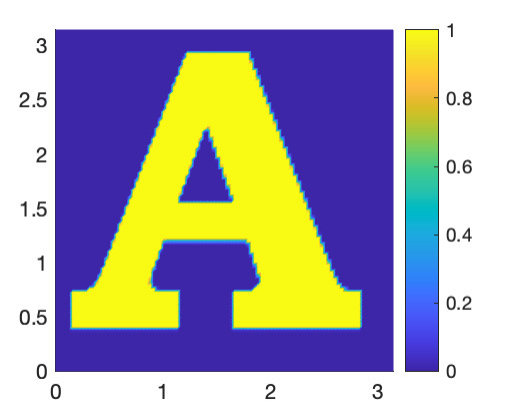}
	\caption{ground truth}
\end{subfigure}
\begin{subfigure}{0.32\textwidth}
	\includegraphics[width =\linewidth]{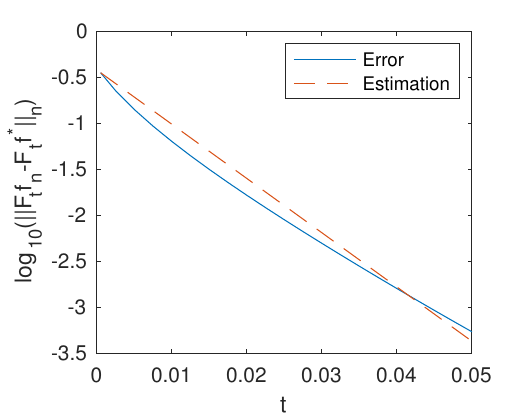}
	\caption{fitted slope: -131.02 predicted slope: -135.81}
\end{subfigure}\\
	 	\begin{subfigure}{0.32\textwidth}
	\includegraphics[width =\linewidth]{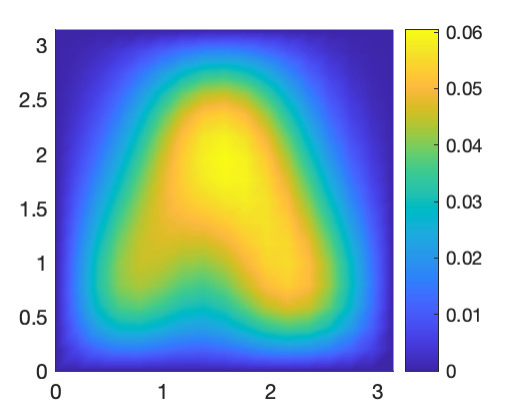}
	\caption{Inversion with $\lambda_0=1$}
 \end{subfigure}
 \begin{subfigure}{0.32\textwidth}
	\includegraphics[width =\linewidth]{figures/eg3_step1.eps}
	\caption{Inversion with $\lambda_1=3.01\times 10^{-4}$}
 \end{subfigure}
 \begin{subfigure}{0.32\textwidth}
	\includegraphics[width =\linewidth]{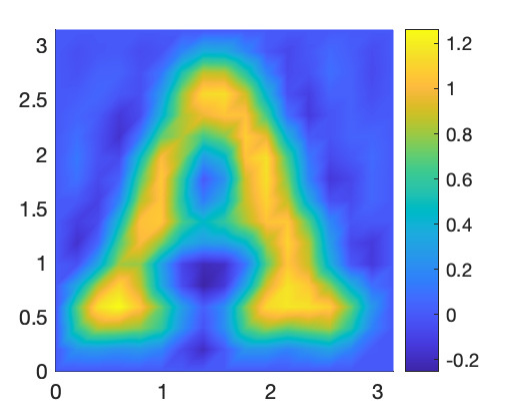}
	\caption{Inversion with $\lambda_2=1.2644\times 10^{-6}$}
\end{subfigure}
	 	\caption{Performance of regularization algorithm when recovering discontinuous $L_2$ initial value (letter A) } 
	 	\centering
	 	\label{fig:A}
	 \end{figure}
	 
\section{Conclusions}

The backward heat conduction problem is a challenging inverse problem. In this paper, we have successfully obtained the stochastic convergence analysis of the regularized solutions for this problem.  By establishing an error estimate for the least-squares regularized minimization problem within the stochastic convergence framework, our analysis demonstrates that the optimal error of the Tikhonov-type least-squares optimization problem depends on the noise level, the number of sensors, and the ground truth. Moreover, we have developed an \textit{a posteriori} algorithm that can find the optimal regularization parameter in the optimization problem without requiring knowledge of the noise level or other prior information. Finally, we demonstrated the accuracy and efficiency of our proposed method through numerical experiments. These results confirm the effectiveness of our method in solving the backward heat conduction problem.

\section*{Acknowledgement}
\noindent 
The research of Z. Zhang is supported by Hong Kong RGC grant projects 17300318 and 17307921, National Natural Science Foundation of China  No. 12171406,  Seed Funding for Strategic Interdisciplinary Research Scheme 2021/22 (HKU), and the Outstanding Young Researcher Award of HKU (2020-21).

\bibliographystyle{plain}

\begin{thebibliography}{10}

\bibitem{agmon}
S.~Agmon.
\newblock {\em Lectures on Elliptic Boundary Value Problems}.
\newblock Van Norstrand, Princeton, NJ, 1965.

\bibitem{AB05}
V.~Ak{\c{c}}elik, G.~Biros, A.~Draganescu, J.~Hill, O.~Ghattas, and
  B.~Waanders.
\newblock Dynamic data-driven inversion for terascale simulations: real-time
  identification of airborne contaminants.
\newblock In {\em SC'05: Proceedings of the 2005 ACM/IEEE Conference on
  Supercomputing}, pages 43--43. IEEE, 2005.

\bibitem{Birman}
M.~Birman and M.~Solomyak.
\newblock Piecewise polynomial approximations of functions of the classes
  $w^k_{\alpha}$.
\newblock {\em Mat. Sb.}, 73:331--355, 1967.

\bibitem{Chen-Zhang}
Z.~Chen, R.~Tuo, and W.~Zhang.
\newblock Stochastic convergence of a nonconforming finite element method for
  the thin plate spline smoother for observational data.
\newblock {\em SIAM Journal on Numerical Analysis}, 56(2):635--659, 2018.

\bibitem{Chen-Zhang2022}
Z.~Chen, W.~Zhang, and J.~Zou.
\newblock Stochastic convergence of regularized solutions and their finite
  element approximations to inverse source problems.
\newblock {\em SIAM Journal on Numerical Analysis}, 60(2):751--780, 2022.

\bibitem{cheng2020backward}
J.~Cheng, Y.~Ke, and T.~Wei.
\newblock The backward problem of parabolic equations with the measurements on
  a discrete set.
\newblock {\em Journal of Inverse and Ill-posed Problems}, 28(1):137--144,
  2020.

\bibitem{chiwiacowsky2003different}
L~Chiwiacowsky and H.~de~Campos~Velho.
\newblock Different approaches for the solution of a backward heat conduction
  problem.
\newblock {\em Inverse Problems in Engng}, 11(6):471--494, 2003.

\bibitem{clark1994quasireversibility}
G.~Clark and S.~Oppenheimer.
\newblock Quasireversibility methods for non-well-posed problems.
\newblock {\em Electronic Journal of Differential Equations}, 1994(8):1, 1994.

\bibitem{Trong2007backward}
D.~Dang and N.~Tran.
\newblock Regularization of a discrete backward problem using coefficients and
  truncated lagrange polynomials.
\newblock {\em Electron. J. Diff. Eqns.}, 51(1):1--14, 2007.

\bibitem{denche2005modified}
M.~Denche and K.~Bessila.
\newblock A modified quasi-boundary value method for ill-posed problems.
\newblock {\em Journal of Mathematical Analysis and Applications},
  301(2):419--426, 2005.

\bibitem{new2}
A.~El~Badia, A.~El~Hajj, M.~Jazar, and H.~Moustafa.
\newblock Lipschitz stability estimates for an inverse source problem in an
  elliptic equation from interior measurements.
\newblock {\em Applicable Analysis}, 95(9):1873--1890, 2016.

\bibitem{Fleckinger}
J.~Fleckinger and M.~Lapidus.
\newblock Eigenvalues of elliptic boundary value problems with an indefinite
  weight function.
\newblock {\em Transactions of the American Mathematical Society},
  295(1):305--324, 1986.

\bibitem{Geer2000}
S.~A. Van~De Geer.
\newblock {\em Applications of Empirical Process Theory}.
\newblock Cambridge university press, 2000.

\bibitem{GER83}
S.~Gorelick, B.~Evans, and I.~Remson.
\newblock Identifying sources of groundwater pollution: an optimization
  approach.
\newblock {\em Water Resources Research}, 19(3):779--790, 1983.

\bibitem{hasanouglu2021introduction}
A.~Hasano{\u{g}}lu and V.~Romanov.
\newblock {\em Introduction to inverse problems for differential equations}.
\newblock Springer, 2021.

\bibitem{new6}
V.~Isakov.
\newblock {\em Inverse problems for partial differential equations}, volume
  127.
\newblock Springer, 2006.

\bibitem{Kabanikhin2008survey}
S.~Kabanikhin.
\newblock Definitions and examples of inverse and ill-posed problems.
\newblock {\em Journal of Inverse and Ill-posed Problems}, 16(1):317--357,
  2008.

\bibitem{Kaipio:2005}
J.~Kaipio and E.~Somersalo.
\newblock {\em Statistical and Computational Inverse Problems}, volume 160.
\newblock Springer, New York, 2005.

\bibitem{Krebs2009backward}
J.~Krebs, A.~Louis, and H.~Wendland.
\newblock Sobolev error estimates and a priori parameter selection for
  semi-discrete tikhonov regularization.
\newblock {\em Journal of Inverse and Ill-posed Problems}, 17(1):845--869,
  2009.

\bibitem{lavrent_ev1986ill}
M.~M. Lavrent'ev, Vladimir~G. Romanov, and S.~P. Shishatski$\breve{\imath}$.
\newblock {\em Ill-posed problems of mathematical physics and analysis},
  volume~64.
\newblock American Mathematical Soc., 1986.

\bibitem{liu2018solving}
J.~Liu and B.~Wang.
\newblock Solving the backward heat conduction problem by homotopy analysis
  method.
\newblock {\em Applied Numerical Mathematics}, 128:84--97, 2018.

\bibitem{L08}
X.~Liu.
\newblock Identification of indoor airborne contaminant sources with
  probability-based inverse modeling methods.
\newblock {\em Boulder, CO: PhD dissertation. University of Colorado at
  Boulder}, 2008.

\bibitem{Muniz1999backward}
W.~Muniz, H.~Velho, and F.~Ramos.
\newblock A comparison of some inverse methods for estimating the initial
  condition of the heat equation.
\newblock {\em Journal of Computational and Applied Mathematics},
  103(1):145--163, 1999.

\bibitem{nelson20}
P.~Nelson and S.~Yoon.
\newblock Estimation of acoustic source strength by inverse methods: {P}art i,
  conditioning of the inverse problem.
\newblock {\em Journal of sound and vibration}, 233(4):639--664, 2000.

\bibitem{NNR98}
G.~Nunnari, A.~Nucifora, and C.~Randieri.
\newblock The application of neural techniques to the modelling of time-series
  of atmospheric pollution data.
\newblock {\em Ecological Modelling}, 111(2-3):187--205, 1998.

\bibitem{Pazy:83}
A.~Pazy.
\newblock {\em Semigroups of Linear Operators and Applications to Partial
  Differential Equations}.
\newblock Springer-Verlag, New York, 1983.

\bibitem{Qian2007backward}
Z.~Qian, C.~Fu, and R.~Shi.
\newblock A modified method for a backward heat conduction problem.
\newblock {\em Applied Mathematics and Computation}, 185(1):564--573, 2007.

\bibitem{renardy2006introduction}
M.~Renardy and R.~Rogers.
\newblock {\em An introduction to partial differential equations}, volume~13.
\newblock Springer Science \& Business Media, 2006.

\bibitem{WANG2023112156}
Z.~Wang, W.~Zhang, and Z.~Zhang.
\newblock A data-driven model reduction method for parabolic inverse source
  problems and its convergence analysis.
\newblock {\em Journal of Computational Physics}, 487:112156, 2023.

\bibitem{WY10}
J.~Wong and P.~Yuan.
\newblock A {FE}-based algorithm for the inverse natural convection problem.
\newblock {\em International Journal for numerical methods in fluids},
  68(1):48--82, 2012.

\end{thebibliography}

\end{document}